\documentclass{amsart}
\usepackage{amssymb}
\usepackage{amsthm}
\newtheorem{theor}{Theorem}[section] 
 \newtheorem{cor}{Corollary}[section]
\theoremstyle{definition} 

\theoremstyle{remark} \newtheorem{rem}{Remark}[section]
\newcommand{\pn}{\par\noindent} \newcommand{\pmn}{\par\medskip\noindent}

\begin{document}
\title{Dual quadrangles in the plane}
\author{Irina Busjatskaja \and Yury Kochetkov}
\date{}
\begin{abstract} We consider quadrangles of perimeter $2$ in the plane
with marked directed edge. To such quadrangle $Q$ a
two-dimensional plane $\Pi\in\mathbb{R}^4$ with orthonormal base
is corresponded. Orthogonal plane $\Pi^\bot$ defines a plane
quadrangle $Q^\circ$ of perimeter $2$ and with marked directed
edge. This quadrangle is defined uniquely (up to rotation and
symmetry). Quadrangles $Q$ and $Q^\circ$ will be called dual to
each other. The following properties of duality are proved: a)
duality preserves convexity, non convexity and self-intersection;
b) duality preserves the length of diagonals; c) the sum of
lengths of corresponding edges in $Q$ and $Q^\circ$ is $1$.
\end{abstract} \email{ibusjatskaja@hse.ru, yukochetkov@hse.ru}
\maketitle

\section{Introduction}
\pn We follow the work \cite{CN} (see also the bibliography
there). Let $Q$ be a quadrangle with perimeter $2$ and with marked
directed edge in plane $\mathbb{R}^2$. It means that we indicate
the first vertex and the direction of going around of $Q$.
\begin{rem} If perimeter of a quadrangle is not $2$, then we made a
dilation with some positive $\alpha$. \end{rem} \pn Let $Q=ABCD$
and $A$ be the first vertex. Vectors $\overline{AB}$,
$\overline{BC}$, $\overline{CD}$ and $\overline{DA}$ we will
consider as complex numbers $z_1$, $z_2$, $z_3$ and $z_4$,
respectively. Then
$$z_1+z_2+z_3+z_4=0,\quad\text{and}\quad |z_1|+|z_2|+|z_3|+|z_4|=2.$$
\begin{rem} The above complex description of $Q$ is invariant with
respect to a translation. \end{rem} \pn  In what follows we will
consider only \emph{non degenerate} quadrangles (with one
exception in Section 4), i.e. quadrangles with non-collinear
successive edges. \pmn Let's define complex numbers
$u_1,u_2,u_3,u_4$ in the following way: a) $u_k^2=z_i$, k=1,2,3,4;
b) $u_1$ we choose arbitrarily; c) the rotation from $u_k$ to
$u_{k+1}$, $k=1,2,3$ is in the same direction as the rotation from
$z_k$ to $z_{k+1}$. Let $u_k=a_k+i\,b_k$, $k=1,2,3,4$, then
$$\sum_k (a_k^2+b_k^2)=2\quad\text{and}\quad
\sum_k[a_k^2-b_k^2)+2i\,a_kb_k]=0,$$ i.e. $\bar
a=(a_1,a_2,a_3,a_4)$ and $\bar b=(b_1,b_2,b_3,b_4)$ are a pair of
orthonormal vectors in $\mathbb{R}^4$. Let $\Pi=\langle \bar
a,\bar b\rangle$ be the linear hull. The two-dimensional plane
$\Pi$ uniquely defines its orthogonal complement --- the
two-dimensional plane $\Pi^\bot$. An orthonormal base $(\bar
c,\bar d)$ of $\Pi^\bot$ in its turn defines a quadrangle
$Q^\circ$ of perimeter $2$, which will be called the quadrangle
\emph{dual} to the quadrangle $Q$. \pmn We will prove the
following properties of the quadrangle duality. \begin{itemize}
\item The dual quadrangle $Q^\circ$ is defined uniquely up to
rotation and reflection (Theorem 2.1). \item The change of the
first vertex and the direction of the going around of $Q$ does not
change the dual quadrangle $Q^\circ$ (Theorem 2.2). \item The
duality preserves: a) convexity (Corollary 5.1); b) non-convexity
(Theorem 5.1); c) self-intersection (Theorem 4.1). \item The sum
of lengths of corresponding edges (in the sense of Section 3) of
$Q$ and $Q^\circ$ is $1$ (Theorem 6.1). \item The lengths of the
corresponding diagonals of $Q$ and $Q^\circ$ are equal (Theorem
7.1). \item Parallelograms are self-dual (Theorem 8.1).
\end{itemize}

\section{General remarks}
\pn Our definition of the dual quadrangle $Q^\circ$ is not
strictly correct, because the base $(\bar c,\bar d)$ of $\Pi^\bot$
is not unique: it is defined up to a rotation and up to the order
of base vectors.

\begin{theor} The plane $\Pi^\bot$ uniquely defines the dual quadrangle
up to a rotation and up to a reflection. \end{theor}

\begin{proof} Let a base $(\bar e,\bar f)$ of $\Pi^\bot$ be obtained by
the rotation of $(\bar c,\bar d)$ on an angle $\alpha$. Thus,
$$e_k=c_k\cos(\alpha)-d_k\sin(\alpha),\quad f_k=c_k\sin(\alpha)+
d_k\cos(\alpha),\,k=1,2,3,4,$$ i.e.
$$e_k+i\,f_k=(c_k+i\,d_k)e^{i\alpha}\Rightarrow (e_k+i\,f_k)^2=
(c_k+i\,d_k)^2e^{2i\alpha},\,k=1,2,3,4.$$ Hence, the rotation of
base of $\Pi^\bot$ on angle $\alpha$ implies the rotation of
$Q^\circ$ on angle $2\alpha$. \pmn  Let us now consider the base
$(\bar d,\bar c)$, instead of the base $(\bar c,\bar d)$, then
$$\begin{array}{l}{\rm Re}\left((d_k+i\,c_k)^2\right)=
-{\rm Re}\left((c_k+i\,d_k)^2\right) \\ {\rm
Im}\left((d_k+i\,c_k)^2\right)={\rm Im}\left((c_k+i\,d_k)^2
\right)\end{array} \quad k=1,2,3,4,$$ i.e. this change of base
implies the reflection of $Q^\circ$ with respect to the axis $OY$.
\end{proof} \pn Let $ABCD$ be the quadrangle $Q$, where $A$ is the
first vertex and the order $ABCD$ defines the direction of going
around. Let $(\bar c,\bar d)$ be the base of $\Pi^\bot$ and $KLMN$
be vertices of $Q^\circ$ ($K$ is the first vertex and the order
$KLMN$ defines the direction of going around). \begin{theor} The
dual of quadrangle $Q$ does not depend on the choice of the first
vertex and on the direction of the going around.\end{theor}
\begin{proof} Let us consider the going around of $Q=ABCD$ in
the same direction, but the first vertex be $B$, i.e. $Q=BCDA$.
Complex numbers $z_1,z_2,z_3,z_4$ are the same, but in order
$z_2,z_3,z_4,z_1$. Complex numbers $u_2,u_3,u_4$ are the same, but
the last one may be $u_1$ or $-u_1$. If the last number is $u_1$,
then $\Pi=\langle (a_2,a_3,a_4,a_1),(b_2,b_3,b_4,b_1)\rangle$  and
$\Pi^\bot=\langle (c_2,c_3,c_4,c_1),(d_2,d_3,d_4,d_1)\rangle$.
Thus, if $KLMN$ is the original dual quadrangle, then $LMNK$ is
the new one, but the same. If the last number is $-u_1$, then
$\Pi=\langle (a_2,a_3,a_4,-a_1),(b_2,b_3,b_4,-b_1)\rangle$ and
$\Pi^\bot=\langle (c_2,c_3,c_3,-c_1),(d_2,d_3,d_4,-d_1)\rangle$,
i.e. the result is the same because
$(-c_1-i\,d_1)^2=(c_1+i\,d_1)^2$. \pmn Now let us consider the
going around in the opposite direction, hence, $Q=ADCB$. In this
case we must consider complex numbers $-z_4,-z_3,-z_2,-z_1$ and
their square roots $\pm i\,u_4$, $\pm i\,u_3$, $\pm i\,u_2$, $\pm
i\,u_1$. Thus
$$\Pi=\langle \mp (a_4,a_3,a_2,a_1),\pm (b_4,b_3,b_2,b_1)\rangle$$ and
$$\Pi^\bot=\langle (c_4,c_3,c_2,c_1),(d_4,d_3,d_2,d_1)\rangle,$$ i.e.
$Q^\circ=KNML$. \end{proof}

\begin{rem} The rotation of $Q$ does not change the plane $\Pi$. \end{rem}

\begin{cor} Let $Q^\circ$ is dual to $Q$ and $(Q^\circ)^\circ$ is dual
to $Q^\circ$, then $Q=(Q^\circ)^\circ$ up to a rotation and up to
a reflection. \end{cor}

\section{The main construction}
\pn In this section, using the knowledge of lengths of edges and
angles of the quadrangle $Q$ we will construct the base of the
plane $\Pi^\bot$. \pmn Let $Q=ABCD$ be positioned in the following
way: $A$ is at the origin, $B$ is in the positive real axis, $C$
and $D$ are in the upper half-plane. Let $|AB|=s_1$, $|BC|=s_2$,
$|CD|=s_3$ and $|DA|=s_4$. 4-dimensional vectors $\bar
a=(a_1,a_2,a_3,a_4)$ and $\bar b=(0,b_2,b_3,b_4)$ will be
considered as quaternions $a$ and $b$. Let $v=(0,a_2,a_3,a_4)$,
$|v|^2=a_2^2+a_3^2+a_4^2=1-a_1^2=1-s_1$. We consider quaternion
products $g=a\cdot
v=(-a_2^2-a_3^2-a_4^2,a_1a_2,a_1a_3a_1a_4)=(s_1-1,
a_1a_2,a_1a_3,a_1a_4)$ and $h=b\cdot
v=(0,b_3a_4-b_4a_3,b_4a_2-b_2a_4, b_2a_3-b_3a_4)$. The
corresponding vectors $\bar g$ and $\bar h$  constitute an
orthogonal base of $\Pi^\bot$ (not orthonormal, because
$|g|=|h|=\sqrt{1-s_1}$). \pmn Let $Q$ be a convex quadrangle
\[\begin{picture}(90,75) \put(0,15){\vector(1,0){90}}
\put(87,7){\scriptsize x} \put(15,5){\vector(0,1){65}}
\put(7,65){\scriptsize y} \put(15,15){\line(1,2){15}}
\put(30,45){\line(2,1){30}} \put(60,60){\line(1,-3){15}}
\put(7,7){\scriptsize A} \put(73,5){\scriptsize B}
\put(65,58){\scriptsize C} \put(25,48){\scriptsize D}
\end{picture}\]
\begin{center}{\large Figure 1}\end{center} \pmn Then
$$z_1=s_1,\,z_2=s_2\exp(\pi-\beta_2),\,z_3=s_3\exp(2\pi-\beta_2-\beta_3),\,
z_4=s_4\exp(\pi+\beta_1),$$ where $\angle DAB=\beta_1$, $\angle
ABC=\beta_2$, $\angle BCD=\beta_3$ and $\angle CDA=\beta_4$. Thus,
$$\begin{array}{l}\bar a=[\sqrt{s_1},\sqrt{s_2}\sin(\gamma_2),
-\sqrt{s_3}\cos(\gamma_2+\gamma_3),-\sqrt{s_4}\sin(\gamma_1)],\\
\bar b=[0,\sqrt{s_2}\cos(\gamma_2),\sqrt{s_3}
\sin(\gamma_2+\gamma_3),\sqrt{s_4}\cos(\gamma_1)],\end{array}$$
where $\gamma_k=\beta_k/2$, $k=1,2,3,4$, and
$$\begin{array}{l}\bar g=[s_1-1,\sqrt{s_1s_2}\sin(\gamma_2),
-\sqrt{s_1s_3}\cos(\gamma_2+\gamma_3),-\sqrt{s_1s_4}\sin(\gamma_1)],\\
\bar h=[0,-\sqrt{s_3s_4}\cos(\gamma_4),\sqrt{s_2s_4}
\sin(\gamma_1+\gamma_2),-\sqrt{s_2s_3}\cos(\gamma_3)]
\end{array}\eqno(1)$$ \pn If our quadrangle $Q$ is non convex
\[\begin{picture}(100,95) \put(0,15){\vector(1,0){100}}
\put(15,5){\vector(0,1){80}} \put(98,7){\scriptsize x}
\put(7,79){\scriptsize y} \put(7,7){\scriptsize A}
\put(15,15){\line(4,5){60}} \put(75,15){\line(-1,2){15}}
\put(60,45){\line(1,3){15}} \put(74,7){\scriptsize B}
\put(64,44){\scriptsize C} \put(79,87){\scriptsize D}
\end{picture}\]
\begin{center}{\large Figure 2}\end{center} \pmn
then
$$z_1=s_1,\,z_2=s_2\exp(\pi-\beta_2),\,z_3=s_3\exp(\beta_3-\beta_2),\,
z_4=s_4\exp(180+\beta_1),$$ where $\angle DAB=\beta_1$, $\angle
ABC=\beta_2$, $\angle BCD=\beta_3$ and $\angle CDA=\beta_4$. And
$$\begin{array}{l}\bar a=[\sqrt{s_1},\sqrt{s_2}\sin(\gamma_2),
\sqrt{s_3}\cos(\gamma_3-\gamma_2),-\sqrt{s_4}\sin(\gamma_1)],\\
\bar
b=[0,\sqrt{s_2}\cos(\gamma_2),\sqrt{s_3}\sin(\gamma_3-\gamma_2),
\sqrt{s_4}\cos(\gamma_1)].\end{array}$$ Thus
$$\begin{array}{l} \bar g=[s_1-1,\sqrt{s_1s_2}\sin(\gamma_2),
\sqrt{s_1s_3}\cos(\gamma_3-\gamma_2),-\sqrt{s_1s_4}\sin(\gamma_1)],\\
\bar
h=[0,-\sqrt{s_3s_4}\cos(\gamma_4),\sqrt{s_2s_4}\sin(\gamma_1+\gamma_2),
\sqrt{s_2s_3}\cos(\gamma_3)].\end{array}\eqno(2)$$ \pn If at last
our quadrangle $Q$ is self-intersecting
\[\begin{picture}(100,95) \put(0,10){\vector(1,0){95}}
\put(93,5){\scriptsize x} \put(10,2){\vector(0,1){90}}
\put(3,88){\scriptsize y} \put(2,2){\scriptsize A}
\put(68,2){\scriptsize B} \put(10,10){\line(5,2){75}}
\put(70,10){\line(-3,4){45}} \put(25,70){\line(2,-1){60}}
\put(88,38){\scriptsize D} \put(17,68){\scriptsize
C}\end{picture}\]
\begin{center}{\large Figure 3}\end{center} \pmn
then
$$z_1=s_1,\,z_2=s_2\exp(\pi-\beta_2),\,z_3=s_3\exp(\beta_3-\beta_2),
z_4=s_4\exp(\beta_1-\pi),$$ where $\angle ABC=\beta_2$, $\angle
BCD=\beta_3$, $\angle CDA=\beta_4$, $\angle DAB=\beta_1$. And
$$\begin{array}{l} \bar a=[\sqrt{s_1},\sqrt{s_2}\sin(\gamma_2),
\sqrt{s_3}\cos(\gamma_3-\gamma_2),\sqrt{s_4}\sin(\gamma_1)],\\
\bar
b=[0,\sqrt{s_2}\cos(\gamma_2),\sqrt{s_3}\sin(\gamma_3-\gamma_2),
-\sqrt{s_4}\cos(\gamma_1)].\end{array}$$ Thus,
$$\begin{array}{l} \bar g=[s_1-1,\sqrt{s_1s_2}\sin(\gamma_2),
\sqrt{s_1s_3}\cos(\gamma_3-\gamma_2),\sqrt{s_1s_4}\sin(\gamma_1)],\\
\bar
h=[0,\sqrt{s_3s_4}\cos(\gamma_4),-\sqrt{s_2s_4}\sin(\gamma_1+\gamma_2),
\sqrt{s_2s_3}\cos(\gamma_3)].\end{array}\eqno(3)$$

\section{Self-intersecting quadrangles}
\pn \begin{theor} The quadrangle, dual to a self-intersecting quadrangle,
is also self-intersecting. \end{theor}

\begin{proof} Let $Q$ be a self-intersecting quadrangle $ABCD$
(see Figure 3) and $Q^\circ=KLMN$ --- its dual. As $z_2$ belongs
to the upper half-plane, then $a_2>0$ and $g_2>0$. As
$a_3b_4-a_4b_3<0$, because of the clockwise turn from $z_3$ to
$z_4$, then $h_2>0$ and $(g_2+i\,h_2)^2$ belongs to the upper
half-plane. Thus, $M$ also belongs to the upper half-plane ($K$ is
at origin and $L=(1-s_1,0)$). \pmn As the turn from $z_2$ to $z_3$
is clockwise, then $a_2b_3-a_3b_2<0$ and $h_4>0$. As $g_4>0$, then
$(g_4+i\,h_4)^2$ belongs to the upper half-plane. Thus, $N$
belongs to the lower half-plane (the direction of the vector
$\overline{NK}$ is up), i.e. $Q^\circ$ cannot be convex
--- $M$ and $N$ belong to different half-planes with respect to
$KL$. \pmn Now we will demonstrate that $Q^\circ$ is
self-intersecting. Let us consider the following highly symmetric
quadrangle $Q_0=ABCD$
\[\begin{picture}(90,90) \put(5,15){\vector(1,0){85}}
\put(15,5){\vector(0,1){83}} \put(88,7){\scriptsize x}
\put(9,84){\scriptsize y} \put(15,15){\line(1,1){60}}
\put(15,75){\line(1,0){60}} \put(15,75){\line(1,-1){60}}
\put(7,7){\scriptsize A} \put(73,7){\scriptsize B}
\put(7,73){\scriptsize C} \put(79,73){\scriptsize D}
\end{picture}\]
\begin{center}{\large Figure 4}\end{center} \pmn where all angles
$\beta_k$, $k=1,2,3,4$, are $\pi/3$ (see Figure 3). Its dual
$Q_0^\circ=KLMN$
\[\begin{picture}(120,115) \put(5,55){\vector(1,0){115}}
\put(15,30){\vector(0,1){75}} \put(118,47){\scriptsize x}
\put(8,100){\scriptsize y} \put(15,55){\line(2,-3){30}}
\put(45,10){\line(1,3){30}} \put(75,100){\line(2,-3){30}}
\put(8,47){\scriptsize K} \put(103,47){\scriptsize L}
\put(66,98){\scriptsize M} \put(50,7){\scriptsize N}
\end{picture}\]
\begin{center}{\large Figure 5}\end{center} \pmn is the same
quadrangle, rotated clockwise on $\pi/3$. Let us assume that there
exists a self-intersecting quadrangle $Q_1$ with non-convex dual
(see below)
\[\begin{picture}(50,50) \put(0,5){\line(1,2){20}}
\put(0,5){\line(1,1){10}} \put(10,15){\line(1,0){40}}
\put(20,45){\line(1,-1){30}} \end{picture}\]
\begin{center}{\large Figure 6}\end{center} \pmn Let $Q_t$,
$0\leqslant t\leqslant 1$, be a continuous family of
non-degenerate self-intersecting quadrangles, that connects $Q_0$
with $Q_1$. We construct this family by moving vertices $B$, $C$
and $D$. Then the continuous family $Q_t^\circ$ connects
$Q_0^\circ$ with $Q_1^\circ$. Hence, for some $\alpha$,
$0<\alpha<1$, the dual quadrangle $Q^\circ_\alpha$ must be
degenerate:
\[\begin{picture}(200,80) \put(10,40){\line(1,0){60}}
\put(10,10){\line(0,1){30}} \put(10,10){\line(1,2){30}}
\put(40,70){\line(1,-1){30}} \put(2,38){\scriptsize K}
\put(15,8){\scriptsize N} \put(73,38){\scriptsize L}
\put(45,68){\scriptsize M} \put(35,5){\small $Q^\circ_t$}

\put(120,40){\line(1,0){60}} \put(110,20){\line(1,2){25}}
\put(135,70){\line(3,-2){45}} \put(112,40){\scriptsize K}
\put(115,18){\scriptsize N} \put(183,38){\scriptsize L}
\put(140,70){\scriptsize M} \put(145,5){\small $Q^\circ_\alpha$}
\end{picture}\] \begin{center}{\large Figure 7}\end{center}
\pmn Now we will consider quadrangles $(Q_t^\circ)^\circ$ and take
the limit for $t\to\alpha$. Let us consider the quadrangle
$Q^\circ_t$ in the left part of Figure 7 and let (with some
abusing of the notation) $|KL|=s_1$, $|LM|=s_2$, $|MN|=s_3$,
$|NK|=s_4$, $\angle NKL=\beta_1$, $\angle KLM=\beta_2$, $\angle
LMN=\beta_3$, $\angle MNK=\beta_4$. Then
$$\begin{array}{l} \bar a_t=[\sqrt{s_1},\sqrt{s_2}\sin(\gamma_2),
-\sqrt{s_3}\cos(\gamma_2+\gamma_3),
\sqrt{s_4}\sin(\gamma_2+\gamma_3-\gamma_4)],\\
\bar
b_t=[0,\sqrt{s_2}\cos(\gamma_2),\sqrt{s_3}\sin(\gamma_2+\gamma_3),
\sqrt{s_4}\cos(\gamma_2+\gamma_3-\gamma_4)].\end{array}$$ When
$t\to\alpha$, then $Q_t^\circ$ on the left (Figure 7) is
transformed into $Q_\alpha^\circ$ on the right, with angles
$\angle MKL=\tilde\beta_1$, $\angle KLM=\tilde\beta_2$ and $\angle
LMK=\tilde\beta_3$. When $t\to\alpha$, then $\beta_4\to 0$,
$\beta_2\to\tilde\beta_2$, $\beta_3\to\tilde\beta_3$ and
$\beta_1\to \pi-\tilde\beta_1$. Hence, $$\begin{array}{l} \bar
a_\alpha=[\sqrt{s_1},\sqrt{s_2}\sin(\tilde\gamma_2),
-\sqrt{s_3}\sin(\tilde\gamma_1),\sqrt{s_4}\cos(\tilde\gamma_1)],\\
\bar
b_\alpha=[0,\sqrt{s_2}\cos(\tilde\gamma_2),\sqrt{s_3}\cos(\tilde\gamma_1),
\sqrt{s_4}\sin(\tilde\gamma_1)].\end{array}$$ Thus,
$$\begin{array}{l}\bar g_\alpha=[\sqrt{1-s_1},\sqrt{s_1s_2}\sin(\tilde\gamma_2),
-\sqrt{s_1s_3}\sin(\tilde\gamma_1),\sqrt{s_1s_4}\cos(\tilde\gamma_1)],\\
\bar
h_\alpha=[0,\sqrt{s_3s_4},-\sqrt{s_2s_4}\cos(\tilde\gamma_1+\tilde\gamma_2),
-\sqrt{s_2s_3}\sin(\tilde\gamma_1+\tilde\gamma_2)].\end{array}$$
The quadrangle, constructed with the use of vectors $\bar
g_\alpha$ and $\bar h_\alpha$, belongs to our family $Q_t$
(Corollary 2.1). Now let us consider complex numbers
$(g_\alpha)_3+i(h_\alpha)_3$, $(g_\alpha)_4+ i(h_\alpha)_4$ and
compute the product
\begin{multline*}\dfrac{(h_\alpha)_3}{(g_\alpha)_3}\cdot
\dfrac{(h_\alpha)_4}{(g_\alpha)_4}=\\=
\dfrac{\sqrt{s_2s_4}\cos(\tilde\gamma_1+\tilde\gamma_2)}
{\sqrt{s_1s_3}\sin(\tilde\gamma_1)}\cdot
-\dfrac{\sqrt{s_2s_3}\sin(\tilde\gamma_1+\tilde\gamma_2)}
{\sqrt{s_1s_4}\cos\tilde(\gamma_1)}=-\dfrac{s_2\sin(\tilde\beta_3)}
{s_1\sin(\tilde\beta_1)} =-1\end{multline*} (because in triangle
the ratio of an edge to the sine of the opposite angle is constant
and equal to the diameter of the circumscribed circle). As this
product is $-1$ then the corresponding vectors are orthogonal.
Thus the squaring of this complex numbers produces collinear
vectors. Hence, the quadrangle $(Q^\circ_\alpha)^\circ$ is
degenerate. But it cannot be so, because it belongs to our
non-degenerate family. \end{proof}

\section{Non-convex quadrangles}
\pn \begin{theor} If $Q$ is non-convex quadrangle, then its dual
is also non convex. \end{theor} \begin{proof} Let $Q$ be a
quadrangle in Figure 2. As $z_2$ belongs to the upper half-plane,
then $a_2>0$, thus $g_2>0$. As $a_3b_4-a_4b_3>0$, because of the
counter clockwise turn from $z_3$ to $z_4$, then $h_2<0$, i.e.
$(g_2+i\,h_2)^2$ belongs to the lower half-plane. Thus, the vertex
$M$ lies in the lower half-plane. \pmn As $z_4$ belongs to the
lower half-plane, then $a_4<0$, $b_4>0$, thus $g_4<0$. As
$a_2b_3-a_3b_2<0$, because of the clockwise turn from $z_2$ to
$z_3$, then $h_4>0$, i.e. $(g_4+i\,h_4)^2$ belongs to the lower
half-plane. Thus, the vertex $N$ lies in the upper half-plane,
i.e. vertices $M$ and $N$ lie in different half-planes with
respect to edge $KL$. Hence, the quadrangle $KLMN$ cannot be
convex. But by Theorem 4.1. it cannot be self-intersecting, so it
is non-convex.
\end{proof}

\begin{cor} The dual to a convex quadrangle is also convex. \end{cor}

\section{Edges}
\pn \begin{theor} Let $Q=ABCD$ be a convex quadrangle and
$Q^\circ=KLMN$ be its dual. Then $|AB|+|KL|=1$, $|BC|+|LM|=1$,
$|CD|+|MN|=1$ and $|DA|+|NK|=1$. \end{theor} \begin{proof} As
$|g|=|h|=\sqrt{1-s_1}$, then have to prove that
$|(g_2+i\,h_2)^2|=g_2^2+h_2^2=(1-s_1)(1-s_2)$. Let $|AC|=l$, then
$$\begin{array}{l}
g_2^2+h_2^2=s_1s_2\sin^2(\gamma_2)+ s_3s_4\cos^2(\gamma_4)=\\
\hspace{5mm}[s_1s_2-s_1s_2\cos(\beta_2)
+s_3s_4+s_3s_4\cos(\beta_4)]/2=\\ \hspace{10mm}
\left[s_1s_2+(l^2-s_1^2-s_2^2)/2)+s_3s_4-(l^2-s_3^2-s_4^2)/2\right]/2=\\
\hspace{15mm}[(s_3+s_4)^2-(s_1-s_2)^2]/4=
[(s_3+s_4+s_1-s_2)(s_3+s_4-s_1+s_2)]=\\ \hspace{40mm}
=(1-s_2)(1-s_1).\end{array}$$ The same reasoning proves that
$|NK|=1-s_4$. As perimeters of $Q$ and $Q^\circ$ are 2, then
$|MN|=1-s_3$. \end{proof}
\begin{rem} The same reasoning proves the theorem for non-convex and
self-intersecting quadrangles. \end{rem}

\section{Diagonals}
\pn \begin{theor} Let $Q=ABCD$ be a convex quadrangle and
$Q^\circ=KLMN$ be its dual, then $|AC|=|KM|$ and $|BD|=|LN|$, i.e.
the duality preserves lengths of diagonals.\end{theor}
\begin{proof} Let $l=|AC|=|z_1+z_2|=|z_3+z_4|$. We will prove,
that $|g_1^2+(g_2+i\,h_2)^2|=(1-s_1)l$. At first we will find the
real part of the complex number $(g_2+i\,h_2)^2$:
\begin{multline*}\text{Re}(g_2+i\,h_2)^2=g_2^2-h_2^2=
s_1s_2\sin^2(\gamma_2)-s_3s_4\cos^2(\gamma_4)=\\
[s_1s_2(1-\cos(\beta_2)-s_3s_4(1+\cos(\beta_4)]/2=\\=
[2s_1s_2+l^2-s_1^2-s_2^2-2s_3s_4+l^2-s_3^2-s_4^2]/4=\\
=[2l^2-(s_1-s_2)^2-(s_3+s_4)^2]/4.\end{multline*} Now the real
part of $g_1^2+(g_2+i\,h_2)^2$ is
\begin{gather*}(1-s_1)^2+
[2l^2-(s_1-s_2)^2-(s_3+s_4)^2]/4=\\=
[4(1-s_1)^2+2l^2-(s_1-s_2)^2-(s_3+s_4)^2]/4=\\=
[2(1-s_1)^2+2l^2+(1-2s_1+s_2)(1-s_2)+\\+(1-s_1+s_3+s_4)(1-s_1-s_3-s_4)]/4=\\
=[2(1-s_1)^2+2l^2+(1-2s_1+s_2)(1-s_2)+(s_2-1)(3-2s_1-s_2)]/4=\\
=[2(1-s_1)^2+2l^2-2(1-s_2)^2]/4=[l^2-(s_1-s_2)(s_3+s_4)]/2.\end{gather*}
Now we will find the square of the imaginary part of
$g_1^2+(g_2+i\,h_2)^2$:
\begin{multline*}4s_1s_2s_3s_4\sin^2(\gamma_2)\cos^2(\gamma_4)=\\=
s_1s_2(1-\cos(\beta_2)s_3s_4(1+\cos(\beta_4))=\\=
(2s_1s_2+l^2-s_1^2-s_2^2)(2s_3s_4+s_3^2+s_4^2-l^2)/4=\\=
(l^2-(s_1-s_2)^2)((s_3+s_4)^2-l^2)/4.\end{multline*} At last we
can find $|g_1^2+(g_2+i\,h_2)^2|^2$:
\begin{multline*} [(l^2-(s_1-s_2)(s_3+s_4))^2+
(l^2-(s_1-s_2)^2)((s_3+s_4)^2-l^2)]/4=\\=
[l^2(-2(s_1-s_2)(s_3+s_4)+(s_3+s_4)^2+(s_1-s_2)^2)]/4=\\=l^2
[(s_3+s_4-s_1+s_2)^2]/4=l^2(1-s_1)^2.\end{multline*} Analogously,
we can prove that $|g_1^2+(g_4+i\,h_4)^2|=(1-s_1)\cdot |BD|$.
\end{proof} \begin{rem} The statement of this theorem is also valid for
non-convex and self-intersecting quadrangles. The reasoning is the
same.
\end{rem}

\section{Special cases}
\begin{theor} The dual to a trapezoid is a trapezoid. \end{theor}
\begin{proof} Let $Q=ABCD$ be a trapezoid, where $AB\parallel CD$:
\[\begin{picture}(110,75) \put(0,15){\vector(1,0){110}}
\put(15,5){\vector(0,1){70}} \put(107,7){\scriptsize x}
\put(7,70){\scriptsize y} \put(15,15){\line(2,3){30}}
\put(45,60){\line(1,0){30}} \put(75,60){\line(1,-3){15}}
\put(7,7){\scriptsize A} \put(88,7){\scriptsize B}
\put(79,59){\scriptsize C} \put(37,59){\scriptsize D}
\end{picture}\]
\begin{center}{\large Figure 8}\end{center}
\pn Here $z_3$ is a negative real number, hence, $u_3=\alpha\,i$,
$\alpha>0$, hence, $g_3=0$, hence $(g_3+i\,h_3)^2$  is a negative
real number. \end{proof}
\begin{theor} The dual to a parallelogram is the same
parallelogram. \end{theor} \begin{proof} Let $Q=ABCD$ be a
parallelogram. As $|AB|+|BC|=1$, then $|KL|=|BC|$ and $|LM|=|AB|$.
It remains to note that $|AC|=|KM|$. \end{proof}

\section{The geometric construction}
\pn Given a convex quadrangle $Q$ it is easy to construct the dual
$Q^\circ$, using ruler and compass. \pmn Let $Q=ABCD$ be a convex
quadrangle
\[\begin{picture}(105,105) \put(15,45){\line(1,-1){30}}
\put(15,45){\line(1,1){45}} \put(45,15){\line(3,2){45}}
\put(60,90){\line(2,-3){30}} \put(7,42){\scriptsize A}
\put(64,89){\scriptsize B} \put(93,42){\scriptsize C}
\put(49,9){\scriptsize D} \qbezier[40](15,45)(52,45)(90,45)
\end{picture}\]
\begin{center}{\large Figure 9}\end{center}
\pmn with diagonal $AC$. Let $|AB|=s_1$, $|BC|=s_2$, $|CD|=s_3$
and $|DA|=s_4$. Using compass we construct the point $B_1$: a) it
is in the same half-plane (with respect to $AC$) as point $B$; b)
$|B_1A|=(s_2+s_3+s_4-s_1)/2$; c) $|B_1C|=(s_1+s_3+s_4-s_2)/2$. In
the same way we construct the point $D_1$: a)  it is in the same
half-plane (with respect to $AC$) as point $D$; b)
$|D_1A|=(s_1+s_2+s_3-s_4)/2$; c) $|D_1C|=(s_1+s_2+s_4-s_3)/2$.
Then $AB_1CD_1$ will be the required dual.

\vspace{7mm}

\begin{thebibliography}{99}

\bibitem{CN} J.Cantarella, T.Needham, C.Shonkwiler \& G.Stewart, \emph{Random
triangles and polygons in the plane}, The American Mathematical
Monthly, 126(2), 2019, 113-134.

\end{thebibliography}
\end{document}